\documentclass{amsart}
\usepackage{amsfonts,amssymb,tikz,lineno}
\usepackage[non-sorted-cites]{amsrefs}

\newcommand\Z{\mathbb{Z}}
\newcommand\N{\mathbb{N}}
\newcommand\R{\mathbb{R}}
\newcommand\Grig{{G_{012}}}
\newcommand\Erschler{{G_{01}}}

\def\pair<#1>{{\langle\!\langle}#1{\rangle\!\rangle}}

\newtheorem{lemma}{Lemma}[section]
\newtheorem{proposition}[lemma]{Proposition}
\newtheorem{theorem}[lemma]{Theorem}
\newtheorem{corollary}[lemma]{Corollary}
\newtheorem{example}[lemma]{Example}

\newtheorem{maintheorem}{Theorem}

\newtheorem*{mainquestion}{Question}
\theoremstyle{definition}
\newtheorem{remark}[lemma]{Remark}

\begin{document}
\title{Growth of permutational extensions}
\author{Laurent Bartholdi}
\author{Anna Erschler}
\date{14 December 2010; revised 31 July, 9 August and 19 October 2011}
\thanks{The work is supported by the ERC starting grant GA 257110
  ``RaWG'' and the Courant Research Centre ``Higher Order Structures''
  of the University of G\"ottingen}

\begin{abstract}
  We study the geometry of a class of group extensions, containing
  permutational wreath products, which we call ``permutational
  extensions''. We construct for all $k\in\N$ a torsion
  group $K_k$ with growth function
  \[ v_{K_k}(n)\sim\exp(n^{1-(1-\alpha)^k}),\quad 2^{3-3/\alpha}+2^{2-2/\alpha}+2^{1-1/\alpha}=2,\]
  and a torsion-free group $H_k$ with growth function
  \[v_{H_k}(n)\sim\exp(\log(n)n^{1-(1-\alpha)^k}).
  \]
  These are the first examples of groups of intermediate growth for
  which the asymptotics of their growth function is known.

  We construct a group of intermediate growth that contains the group
  of finitely supported permutations of a countable set as a
  subgroup. This gives the first example of a group of intermediate
  growth containing an infinite simple group as a subgroup.
\end{abstract}
\maketitle

\section{Introduction}

Let $G$ be a finitely generated group. One of its fundamental
invariants is the group's \emph{growth}, studied since the
1950's~\cites{krause:growth,svarts:growth}: choose a finite set $S$
generating $G$ as a monoid, and define its \emph{growth function}
\[v_{G,S}(n)=\#\{g\in G\mid g=s_1\dots s_m\text{ for some $m\le n$ and
  $s_i\in S$}\}.
\]
This function depends on $S$, but only mildly: given two functions $f,
g: \N \to \R_+$, we say that $g$ is asymptotically smaller than $f$,
and write $g \preceq f$, if there exist $C>0$ such $g(n) \le f(C n)$
for all sufficiently large $n$. We say that $f$ is asymptotically
equivalent to $g$, and write $f \sim g$, if $f \preceq g$ and $g
\preceq f$.  The asymptotic equivalence class of $v_{G,S}$ is
then independent of $S$, and is written $v_G$.

More information on growth of groups may be found
in~\cite{harpe:ggt}*{Chapters~VI--VIII} and in the
monograph~\cite{mann:growth}.


By a famous theorem of Gromov~\cite{gromov:nilpotent}, the growth
function is at most polynomial if and only if $G$ is virtually
nilpotent. Furthermore, in that case, there exists an integer $d$ such
that $v_G\sim n^d$; this fact is usually attributed to Guivarc'h and
to Bass~\cite{bass:nilpotent}, see also~\cite{harpe:ggt}*{VII.26}.

On the other hand, if $G$ is linear, or solvable, or word-hyperbolic,
then, unless $G$ is virtually nilpotent, $v_G\sim\exp(n)$, so all such
groups have asymptotically equivalent growth.

Milnor asked in~\cite{milnor:5603} whether $v_G$ could take values
strictly between polynomial and exponential functions; such groups
would be called groups of \emph{intermediate growth}. This question
was answered positively by Grigorchuk in
1983~\cites{grigorchuk:growth,grigorchuk:gdegree,grigorchuk:kyoto}.
Grigorchuk constructed a large class of groups of intermediate growth,
showing in particular that for any subexponential function $v(n)$
there exists a group of intermediate growth whose growth function is
greater than $v(n)$ for infinitely many $n$. Essentially the same
argument shows that for any subexponential $v(n)$ there exists a group
which is a direct sum of two Grigorchuk groups such that $v_{G,S}(n)
\ge v(n)$ for all sufficiently large $n$~\cite{erschler:degrees}.
Another result of Grigorchuk is the existence of groups with
incomparable growth functions.

There was, up to now, no group of intermediate growth for which the
asymptotics of its growth function is known. On the other hand, it was
known since the early 1970's that semigroups can have intermediate
growth, and some growth functions were computed explicitly: Govorov
shows in~\cite{MR0318199} that the semigroup
\[\langle x,y\mid x^iy^{i-1}x^i=y^ix^{i-1}y^i=0\quad\forall i\ge2\rangle_+\]
has growth function $\sim \exp(\sqrt n)$. The precise asymptotics of
the growth of more semigroups (including $2\times 2$ matrix
semigroups, and relatively free semigroups) were computed by
Lavrik~\cite{lavrik:interm}, Shneerson~\cite{shneerson:interm},
Okninski~\cite{MR2001g:20076}, and Reznykov and
Sushchanskii~\cite{bartholdi-r-s:interm2x2}. The following growth
types occur among these examples: $n^{\log n}$, $\exp(\sqrt n)$, and
$\exp(\sqrt n/\log n)$.

The most studied example of group of intermediate growth is the
\emph{first Grigorchuk group} $\Grig$. The best known estimates for
this group's growth are as follows: let $\eta$ be the real root of the
polynomial $X^3+X^2+X-2$, and set
$\alpha=\log(2)/\log(2/\eta)\approx0.7674$. Then
\begin{equation}\label{eq:grigupper}
  \exp(n^{0.5153})\precsim v_\Grig\precsim\exp(n^\alpha).
\end{equation}
For the lower bound, see~\cite{bartholdi:lowerbd} and
Leonov~\cite{leonov:lowerbd2}; for the upper bound,
see~\cite{bartholdi:upperbd}.

\subsection{Main results} Up to now no growth function $\not\sim
n^d,\exp(n)$ had been determined. In this paper, we construct the
first examples of groups of intermediate growth for which the
asymptotics of the growth function is known:
\begin{maintheorem}\label{mainthm:A}
\begin{enumerate}
\item For every $k\in\N$ there exists a finitely generated torsion
  group $K_k$ with growth
  \[v_{K_k}\sim\exp(n^{1-(1-\alpha)^k}).
  \]
\item For every $k\in\N$ there exists a finitely generated
  torsion-free group $H_k$ with growth
  \[v_{H_k}\sim\exp(\log(n)n^{1-(1-\alpha)^k}).\]
\end{enumerate}  
\end{maintheorem}
In fact, the group $H_1$ that we consider had already been studied by
Grigorchuk in~\cite{grigorchuk:gdegree}, who showed that it is
torsion-free and has subexponential growth, though he did not compute
its growth function.


A classical result due to Higman, Neumann and Neumann states that any
countable group can be embedded into a finitely generated
group~\cite{higman-n-n:embed}. Given a countable group,
it is natural to ask
which extra properties that finitely generated group may possess; and,
in the case that interests us, when that finitely generated group may
be taken of subexponential growth. We show:

\begin{maintheorem}[= Theorem~\ref{thm:embedSinfty}]\label{thm:C}
  The group of finitely supported permutations of $\Z$ embeds as a
  normal subgroup in a group of intermediate growth.
\end{maintheorem}

The theorem shows, in particular, that some groups of intermediate
growth contain as a subgroup an infinite simple group. In our examples
such infinite simple subgroups are not finitely generated. An open
question due to Grigorchuk asks whether any infinite finitely
generated simple group has exponential growth. All known example of
finitely generated infinite simple groups are non-amenable.

\begin{mainquestion}
  Does every countable group locally of subexponential growth embed in
  a finitely generated group of subexponential growth?
\end{mainquestion}

\subsection{Outline of the approach}
Our strategy may be summarized as follows. We consider groups $B,G$
with respective generating sets $R,S$, such that: $S$ is finite; $G$
acts on $B$; the action of $G$ on $B$ permutes $R$; and $R$ has
finitely many $G$-orbits. Typically, $G$ will act on a set $X$, and
$B$ will be either: a direct sum $\sum_XA$ of a finitely generated
group; a group of permutations of $X$; or a free Abelian, nilpotent
etc.\ group generated by a set in bijection with $X$. We then take the
semidirect product $W$ of $B$ with $G$. It turns out that the growth
of $W$ is well controlled by the growth of ``inverted orbits'' of $G$
on $X$. By definition, the inverted orbit of a word $w=g_1\dots g_n$
over $G$ is $\{x_0g_1\dots g_n,x_0g_2\dots g_n,\dots,x_0g_n\}$ for a
basepoint $x_0\in X$.

The idea behind the fact that groups of Grigorchuk have intermediate
growth is a certain contraction property. For some of Grigorchuk's
groups $G$, including the first one, this property states that there
exist an injective homomorphism $\psi=(\psi_1, \dots, \psi_d)$ from a
finite index subgroup $H$ in $G$ to the direct product $G^d$, such
that for some finite generating set $S$ in $G$, some $\eta<1$, some
$C>0$, and some proper norm $|\cdot|$ on $G$, the following holds for
all $h\in H$:
\begin{equation}\label{eq:usualcontraction}
\sum_{i=1}^d |\psi_i(h)| \le \eta|h|+C.
\end{equation}

This implies that the growth of $G$ is bounded from above by
$\exp(n^\beta)$ for a constant $\beta<1$ depending only on $\eta$.

In this paper we introduce and study another contraction property,
related to the sublinear growth of the inverted orbits for a group
action. We prove that this property holds for the action on the
boundary of many Grigorchuk groups, including the first Grigorchuk
group. The property implies that not only these groups, but a large
family of extensions of these groups have subexponential growth. We
mention, however, that for some Grigorchuk groups the inverted orbit
growth is linear (see Example~\ref{ex:expgrowth}).

The contracting property~\eqref{eq:usualcontraction} in the case
$d=1$, deserves special mention; the group $G$ then admits a {\it
  dilatation}. This implies that the growth of $G$ is polynomial.
Many nilpotent groups admit a dilatation. However, no nilpotent group
may satisfy the ``sublinear inverted orbit growth'' property studied in
this paper (see Remark~\ref{rem:nilpotentactions}).

\subsection{Acknowledgments} The authors are grateful to Yves de
Cornulier and Pierre de la Harpe for their helpful remarks on a
preliminary version of this text, and to the referee for his/her
very careful reading of the manuscript.

\section{Permutational wreath products and extensions}
We consider groups $A$, $G$ and a set $X$, such that $G$ acts on $X$
from the right. The \emph{wreath product} $W=A\wr_XG$ is the
semidirect product of $\sum_X A$ with $G$.  The \emph{support} $\sup
f$ of a function $f:X\to A$ consists of those $x\in X$ such that
$f(x)\neq1$.  We describe elements of $\sum_XA$ as finitely supported
functions $X\to A$.  The (left) action of $G$ on $\sum_X A$ is then
defined by $(g\cdot f)(x) = f(xg)$; observe that for $g_1$, $g_2$ in
$G$
\[
(g_1g_2\cdot f)(y) = f(y(g_1g_2)) = f((yg_1)g_2) = (g_2\cdot f)(yg_1) = (g_1\cdot g_2\cdot f)(y).
\]
We have in particular $\sup(g^{-1}\cdot f)=\sup(f)g$.  If $A$ act on a
set $Y$ from the right, then $W$ naturally acts on $X\times Y$ from
the right, by $(x,y)f=(x,yf(x))$ and $(x,y)g=(xg,y)$.

Suppose now that $A$ and $G$ are finitely generated, and that the
action of $G$ on $X$ is transitive.  Fix generating sets $S_A$ and
$S_G$ of $A$ and $G$ respectively, and fix a basepoint $x_0\in X$. The
wreath product is generated by $S=S_A \cup S_G$, in which we identify
$G$ with its image under the embedding $g \to (1,g)$, and identify $A$
with its image under the embedding $a \to (f_a,1)$; here $f_a:X \to
A$ is defined by $f_a(x_0)=a$ and $f(x)=1$ for all $x\ne x_0$. We call
$S$ the \emph{standard generating set} of $W$ defined by $S_A$, $S_G$.
Analogously, if the action of $G$ on $X$ has finitely many orbits,
then $W$ is finitely generated by $S_G\cup (S_A\times (X/G))$.

\subsection{The Cayley graph of a permutational wreath product} The
Cayley graph of $W=(\sum_XA)\rtimes G$ with respect to the generating
set $S$ may be described as follows. Elements of $W$ are written $fg$
with $f\in\sum_XA$ and $g\in G$; multiplication is given by
$(f_1g_1)(f_2g_2)=f_1(g_1\cdot f_2)\,g_1g_2$.

Consider a word $v=s_1s_2 \dots s_\ell$, with all $s_i \in S$, and
write its value in $W$ as $f_vg_v$ with $f\in\sum_XA$ and $g_v\in
G$. Set $u=f_ug_u= s_1 s_2 \dots s_{\ell-1}$. Here $g_u, g_v$ belong
to $G$, and $f_u, f_v:X\to A$.

We consider two cases, depending on whether $s_\ell\in S_A$ or
$s_\ell\in S_G$. If $s_\ell\in S_A$, we have an edge of ``A'' type
from $u$ to $v$. The multiplication formula gives $g_v=g_u$ and
$f_v(x)=f_u(x)$ for all $x\neq x_0g_u^{-1}$, while
$f_v(x_0g_u^{-1})=f_u(x_0g_u^{-1})s_\ell$.  If $s_\ell\in S_G$, we
have an edge of ``G'' type from $u$ to $v$. Then $f_v=f_u$, and
$g_v=g_us_\ell$.

There is an alternative description of the edges in the Cayley graph,
which appears if we write elements of $(\sum_XA)\rtimes G$ in the form
$gf$ with $g\in G$ and $f\in\sum_XA$. Their product is then given by
$(g_1f_1)(g_2f_2)=g_1g_2(g_2^{-1}\cdot f_1)f_2$.

In that notation, if there is an edge of ``A'' type from $u=g_uf_u$ to
$v=g_vf_v$, then we have $g_u=g_v$, and $f_v=f_u$ except at $x_0$
where $f_v(x_0)=f_u(x_0)s_\ell$.  On the other hand, if there is an
edge of ``G'' type from $u$ to $v$, then we have $g_v=g_us_\ell$ and
$f_v=s_\ell^{-1}\cdot f_u$.

For $i=1,\dots,\ell$, set now $g_i=s_i,a_i=1$ whenever $s_i\in S_G$,
and $g_i=1,a_i=s_i$ whenever $s_i\in S_A$. Still writing $v=g_vf_v$,
we then have
\[v=a_1g_1\cdots a_\ell g_\ell=g_1\cdots g_\ell a_1^{g_1\cdots g_\ell}\cdots a_{\ell-1}^{g_{\ell-1}g_\ell}a_\ell^{g_\ell},
\]
so $g_v=g_1 g_2 \cdots g_\ell$ and $f_v=((g_1\cdots g_\ell)^{-1}\cdot
f_{a_1})((g_2\cdots g_\ell)^{-1}\cdot f_{a_2})\cdots (g_\ell^{-1}\cdot
f_{a_\ell})$; we observe that the support of $f_v$ is contained in
$\{x_0 g_\ell,x_0g_{\ell-1}g_\ell,\dots,x_0g_1g_2\cdots g_\ell\}$. In
other words, in order to understand the support of the configuration,
we have to study \emph{inverted orbits} of the action of $G$ on $X$
and the number of distinct points visited by these orbits.

\begin{remark}
  In case $G=X$, there is no difference between counting the number of
  points on the orbits or on the inverted orbits ($x_0=1$, $G$ acts
  on $X$ both from the right and from the left, and the inverted
  orbits for the right action are usual orbits for the left action).
  This is no longer the case if $X\ne G$.
\end{remark}

\begin{remark}
  We might wonder to which degree the geometry of the Cayley graphs of
  $A$ and $G$, and of the Schreier graph of $X$ (the graph with vertex
  set $X$, and an edge from $x$ to $xs$ for all $x\in X$ and generator
  $s$ of $G$), determine the geometry of the wreath product.

  In contrast with the case $X=G$ (``usual'' wreath products), the
  Cayley graph of the permutational wreath product is in no way
  defined by the unmarked Cayley graphs of $A$ and $G$ and the
  Schreier graph of $X$. We will see in the sequel that the following
  may happen: a group $G$ acts on $X_1$ and $X_2$, the unmarked
  Schreier graphs of $X_1$ and $X_2$ are the same, but $A\wr_{X_1}G$
  has exponential growth (for some finite group $A$), and
  $A\wr_{X_2}G$ has intermediate growth (see
  Example~\ref{ex:differentgrowth}).
\end{remark}

\subsection{Inverted orbits}\label{ss:invorb}
We formalize the discussion above as follows. Fix a group $G$ acting
on the right on a set $X$; fix a set $S$ generating $G$ as a monoid;
and fix a basepoint $x_0\in X$. Denote by $S^*$ the set of words over
$S$. For a word $w=w_1\dots w_\ell\in S^*$, its \emph{inverted orbit}
is
\[\mathcal O(w)=\{x_0,x_0 w_\ell,x_0w_{\ell-1}w_\ell,\dots,x_0w_1w_2\dots w_\ell\}.
\]
Its \emph{inverted orbit growth} is
\[\delta(w)=\#\mathcal O(w).\]
The \emph{inverted orbit growth function} of $G$ is the function
\[\Delta(n)=\max\{\delta(w)\mid n=|w|\}.\]
Clearly $\Delta(n)\le n+1$; and, if the orbit of $x_0$ is infinite,
$\Delta(n(n-1)/2)\ge n$, so $\Delta(n)\succsim\sqrt n$. Indeed,
consider a word $w=u_n\dots u_1$ in which $u_i$ is a word of length
$\le i$, chosen such that
$x_0u_i\not\in\{x_0,x_0u_{i-1}^{-1},x_0u_{i-2}^{-1}u_{i-1}^{-1},\dots,x_0u_1^{-1}\dots
u_{i-1}^{-1}\}$; then $\delta(w)\ge n+1$.

The functions $\delta$ and $\Delta$ depend on the choice of $x_0$ and
$S$. However, it is easy to see that their asymptotics do not depend
on the choice of the basepoint $x_0$ and the generating set $S$:
\begin{lemma}\label{le:indepbasepoint}
  If $G$ is finitely generated, then the $\sim$-equivalence class of
  $\Delta$ does not depend on the choice of $S$.

  If $G$ acts transitively on $X$, then the $\sim$-equivalence class
  of $\Delta$ does not depend on the choice of $x_0$.
\end{lemma}
\begin{proof}
  Let $S$ and $S'$ be two finite generating sets for $G$; write each
  element of $S$ as a word over $S'$; and let $C$ be the maximum of
  such lengths. We temporarily write $\Delta_S$ and $\Delta_{x_0,S}$
  to remember the dependence on the choices of $x_0\in X$ and
  $S\subset G$.

  Given $w\in S^*$, let $w'$ be the corresponding rewritten word over
  $S'$. We have $|w'|\le C|w|$ and $\delta_S(w)\le\delta_{S'}(w')$, so
  $\Delta_S(n)\le\Delta_{S'}(Cn)$ for all $n\in\N$, and
  $\Delta_S\precsim\Delta_{S'}$. The reverse inequality gives
  $\Delta_S\sim\Delta_{S'}$, and proves the first part of the lemma.

  Now consider two points $x_0,x_1\in X$, and an element $g\in G$ with
  $x_0g=x_1$, of length $k$. Set $S'=\{s^g =g^{-1}sg \mid s\in S\}$.
  It is clear that $S'$ is a generating set of $G$.  Let $w=w_1w_2
  \dots w_\ell$ be a given word, and consider the word $w^g = (g^{-1}
  w_1 g) (g^{-1} w_2 g) \dots (g^{-1} w_\ell g)$ over the alphabet
  $S'$.  Given $i,j\in\{0,\dots,\ell\}$, observe that $x_0 w_i w_{i+1}
  \dots w_\ell \ne x_0 w_j w_{j+1} \dots w_\ell$ if and only if $x_0
  w_i\dots w_\ell g \ne x_0 w_j\dots w_\ell g$, if and only if
  \[
  x_1 (g^{-1}w_ig)\dots g^{-1}w_\ell g \neq x_1 (g^{-1}w_jg)\dots g^{-1}w_\ell g.
  \]
  This implies $\delta_{x_0,S}(w)=\delta_{x_1,S'}(w^g)$. Therefore, we
  have $\Delta_{x_0,S}(n)=\Delta_{x_1,S'}(n)$; and the first part of
  the lemma gives $\Delta_{x_0,S}\sim\Delta_{x_1,S}$.
\end{proof}

\section{Self-similar groups}
Below we recall the definition of some of Grigorchuk's groups. They
are groups acting on a rooted tree. The first Grigorchuk group belongs
to a smaller class of self-similar groups.  We fix our notation for
such groups; for more information on self-similar groups, see
Nekrashevych's book~\cite{nekrashevych:ssg}. Fix an integer $d\ge2$
called the \emph{degree}. Words
$q=q_1\dots q_n\in\{\mathsf0,\dots,\mathsf{d-1}\}^*$ form the vertex set of a
rooted regular tree $\mathcal T$, with root the empty word; and
$q_1\dots q_n$ connected by an edge to $q_1\dots q_{n-1}$.

A \emph{self-similar group} is, by definition, a group
\emph{presented} by a map, called the \emph{wreath recursion},
$\psi:G\to G\wr\mathfrak S_d$, from $G$ to its permutational wreath
product with the symmetric group $\mathfrak S_d$. We write images
under $\psi$ in the form
\[\psi(g)=\pair<g_0,\dots,g_{d-1}>\pi\text{ with }(g_0,\dots,g_{d-1})\in G^d\text{ and }\pi\in\mathfrak S_d.\]
The wreath recursion $\psi$ defines an action of $G$ by isometries on
$\mathcal T$, as follows. Consider $g\in G$ and $q=q_1\dots
q_n\in\mathcal T$. If $n=0$, then $qg=q$. Otherwise, compute
$\psi(g)=\pair<g_0,\dots,g_{d-1}>\pi$, and set inductively $qg=(q_1^\pi)(q_2\dots q_n)g_{q_1}$.

When a self-similar group is given by its wreath recursion, it is
assumed that the action on $\mathcal T$ is faithful; namely, the group
$G$ \emph{defined} by the wreath recursion
$\psi:\Gamma\to\Gamma\wr\mathfrak S_d$ is the quotient $G$ of $\Gamma$
by the kernel of $\Gamma$'s action on $\mathcal T$.  We then drop
`$\psi$' from the notation, and write the wreath recursion on $G$ in
the form
\[g=\pair<g_0,\dots,g_{d-1}>\pi\text{ or }g=\pi\pair<g_0,\dots,g_{d-1}>.\]

The boundary $\partial\mathcal T$ of $\mathcal T$ consists of infinite
sequences; its elements are called \emph{rays}. If $G$ is a
self-similar group acting on $\mathcal T$, then $G$ also acts on
$\partial\mathcal T$. Mainly, the action of $G$ we will be interested
in is that on a ray orbit $\rho G$.

\subsection{The first Grigorchuk group}
An important example of self-similar group was extensively studied by
Grigorchuk~\cite{grigorchuk:(un)solved}. It may be defined by its
wreath recursion as the $4$-generated group $\Grig$ with generators
$\{a,b,c,d\}$; if $\varepsilon$ denote the transposition
$(\mathsf0,\mathsf1)$,
\begin{equation}\label{eqgrig}
\psi:a\mapsto\pair<1,1>\varepsilon,\quad b\mapsto\pair<a,c>,\quad
c\mapsto\pair<a,d>,\quad d\mapsto\pair<1,b>
\end{equation}
Grigorchuk proved in~\cite{grigorchuk:burnside} that $\Grig$ is an
infinite, finitely generated torsion group; and,
in~\cite{grigorchuk:growth}, that $\Grig$ is a group of intermediate
growth.

A presentation of $\Grig$ by generators and relators was given by
Lysionok~\cite{lysionok:pres}: define the endomorphism $\sigma$ of
$\{a,b,c,d\}^*$ by
\begin{equation}\label{eq:sigma}
  \sigma:a\mapsto aca,\quad b\mapsto d,\quad c\mapsto b,\quad d\mapsto c.
\end{equation}
Then
\begin{equation}\label{eq:lysionok}
\Grig=\langle a,b,c,d\mid a^2,b^2,c^2,d^2,bcd,\sigma^n([d,d^a]),\sigma^n([d,d^{(ac)^2a}])\text{ for all }n\in\N\rangle.
\end{equation}

The notation $\Grig$ arises as follows: Grigorchuk defined a continuum
of groups $G_\omega$, for $\omega\in\{0,1,2\}^\N$. The first
Grigorchuk group is in fact the group $G_{(012)^\infty}$ defined by a
periodic sequence $\omega$.

\subsection{Another Grigorchuk group}
By $\Erschler$ we denote in the sequel the following Grigorchuk group.
It is the group generated by $a,b,c,d$ and given by the recursion
\begin{equation}\label{eqgrigdi}
\psi:a\mapsto\pair<1,1>\varepsilon,\quad b\mapsto\pair<1,c>,\quad
c\mapsto\pair<a,b>,\quad d\mapsto\pair<a,d>
\end{equation}
In contrast with the first Grigorchuk group, $\Erschler$ contains
elements of infinite order.  Indeed, the subgroup generated by
$\{a,d\}$ is isomorphic to the infinite dihedral group. Furthermore,
the infinite-order element $ad$ acts freely on the boundary
$\partial\mathcal T$ of the tree on which $\Erschler$
acts~\cite{grigorchuk:gdegree}*{proof of Lemma~9.10}.  If we denote by
$\rho$ the ray $\mathsf1^\infty$, we then have for all integers $m\ne
n$
\begin{equation}\label{eq:freediadicorbit}
  \rho (ad)^n \ne \rho (ad)^m.
\end{equation}
The group $\Erschler$ has intermediate
growth~\cite{erschler:boundarysubexp}; the best known lower and upper
bounds are, for all $\epsilon>0$,
\[\exp(n/\log^{2+\epsilon}(n))\precsim v_\Erschler(n)\precsim\exp(n/\log^{1-\epsilon}(n)).\]

\section{Inverted orbit growth for Grigorchuk's first group}
We fix $\rho=\mathsf1^\infty$ the ray in the binary tree $\mathcal T$.
This ray is fixed by $b,c,d$. We write $\Omega=\{a,b,c,d\}^*$ the set of
words over the standard generators. The length of $w\in \Omega$ is
written $|w|$.  The recursion~\eqref{eqgrig} gives rise to a map
$\Omega\to \mathfrak S_2\times \Omega\times \Omega$, defined by the same
formulas. We still write it in the form
$w\mapsto\varepsilon^s\pair<u,v>$.

We call a word $w\in \Omega$ \emph{pre-reduced}, if it does not
contain two consecutive occurrences of $b,c,d$; the
\emph{pre-reduction} of $w$ is the word obtained from $w$ by deleting
consecutive $bb,cc,dd$ and replacing $bc$ or $cb$ by $d$, $cd$ or $dc$
by $b$, and $db$ or $bd$ by $c$. These operations do not change the
image of the word in $\Grig$. Recall that, for a word $w=w_1\dots
w_n\in \Omega$, we defined
\[\delta(w)=\#\{\rho w_{i+1}\dots w_n\mid i=0,\dots,n\}.\]

\begin{lemma}
  $\delta(w)=\delta(\text{its pre-reduction})$.
\end{lemma}
\begin{proof}
  Consider a subword $w_jw_{j+1}$ of $w$ consisting only of $b,c,d$'s,
  and let $u$ denote the shorter word obtained by replacing
  $w_jw_{j+1}$ by its value. In computing $\delta(w)$, either $i\le j$ or
  $i>j+1$, in which case $\rho w_i\dots w_n=\rho u_i\dots u_{n-1}$; or
  $i=j+1$, in which case $\rho w_i\dots w_n=\rho w_{j+2}\dots w_n=\rho
  u_{i+1}\dots u_{n-1}$, because $w_j$ and $w_{j+1}$ fix $\rho$.
\end{proof}

Let $\eta\approx0.811$ be the real root of the polynomial
$X^3+X^2+X-2$, and consider on $\Omega$ the norm defined by
\[\|a\|=1-\eta^3,\;\|b\|=\eta^3,\;\|c\|=1-\eta^2,\;\|d\|=1-\eta;\]
namely, for a word $w=w_1\dots w_n\in\Omega$ set
$\|w\|=\|w_1\|+\dots+\|w_n\|$.  The norm induced on $\Grig$ by $\|.\|$
was considered by the first author in~\cite{bartholdi:upperbd}.  As we
have already mentioned, the first Grigorchuk group satisfies the
contraction property in Equation~\eqref{eq:usualcontraction}. Here as
a word metric in $\Grig$ one can consider the word metric with respect
to the generating set $a,b,c,d$.  The idea of~\cite{bartholdi:upperbd}
was that if instead of the word metric we consider the norm as defined
above, this leads to a better contraction coefficient $\eta$
in~\eqref{eq:usualcontraction} and a better upper bound of the form
$\exp(n^\alpha)$ for the growth of $\Grig$. In this paper we use this
norm in order to get upper bounds on the growth for extensions of
$\Grig$.

Note that the norm $\|\cdot\|$ and the word length $|\cdot|$ are
equivalent. If $w$ is pre-reduced of length $n$, then it contains at
least $(n-1)/2$ times the letter `$a$'. We may therefore apply the
argument in~\cite{bartholdi:upperbd}*{Proposition~4.2}, which we
reproduce here for completeness, with words \emph{in lieu} of group
elements:
\begin{lemma}[see~\cite{bartholdi:upperbd}*{Proposition~4.2}]\label{lem:contract}
  Consider $w\in \Omega$ pre-reduced, and write
  $w=\varepsilon^s\pair<u,v>$ for $u,v\in \Omega$ and
  $s\in\{0,1\}$. Define $C=\eta \|a\|$. We then have
  \[
   \|u\|+\|v\|  \le \eta\|w\|+C.
  \]
\end{lemma}
\begin{proof}
  Since $w$ is pre-reduced, say of length $n$, it contains at most
  $(n+1)/2$ letters in $\{b,c,d\}$. For each of these letters, we
  consider the corresponding letter(s) in $u$ and $v$ given by the
  wreath recursion. We have
  \begin{align*}
    \eta(\|a\|+\|b\|) &= \|a\|+\|c\|,\\
    \eta(\|a\|+\|c\|) &= \|a\|+\|d\|,\\
    \eta(\|a\|+\|d\|) &= 0+\|b\|.
  \end{align*}
  As $b=\pair<a,c>$ and $aba=\pair<c,a>$, each $b$ in $w$ contributes
  $\|a\|+\|c\|$ to the total weight of $u$ and $v$; the same argument
  applies to $c$ and $d$. Now, grouping together each letters in
  $\{b,c,d\}$ with the `$a$' after it (this is possible for all
  letters except possibly the last), we see that $\eta\|w\|$ is a sum
  of left-hand terms, possibly ${}-\eta\|a\|$; while $\|u\|+\|v\|$ is
  the sum of the corresponding right-hand terms.
\end{proof}

Let us write $\Delta(n)=\max\{\delta(w)\mid n\ge\|w\|\}$; this
function is equivalent to that defined in~\S\ref{ss:invorb}.
We state the following general lemma:
\begin{lemma}\label{lem:general}
  Let $\Delta:\N\to\N$ be a function. Let $\eta\in(0,1)$ and $C$ be such that, for
  all $n\in\N$, there exists $\ell,m\in\N$ with $\ell+m \le \eta n+C$
  and $\Delta(n) \le \Delta(\ell)+\Delta(m)$.  Set
  \[\alpha=\log(2)/\log(2/\eta).\]
  Then we have for all $n\in\N$
  \[\Delta(n)\precsim n^\alpha.\]
\end{lemma}
\begin{proof}
  Define $K=C/(2-\eta)$ and $M=C/(1-\eta)$. We will prove, in
  fact, $\Delta(n)\le L(n-K)^\alpha$ for some constant $L$ and all $n$
  large enough.

  For that purpose, let $L$ be large enough so that $\Delta(n)\le
  L(n-K)^\alpha$ for all $n\le M$. Set $N=K/(1-\alpha)$, define
  $\Delta^*$ by
  \begin{equation}\label{eq:Delta}
    \Delta^*(n)=\begin{cases} L(n-K)^\alpha & \text{ if }n\ge N,\\
      1+(L(N-K)^\alpha-1)n/N & \text{ if }n\le N,\end{cases}
  \end{equation}
  and note that $\Delta^*$ is the convex hull of $1$ and
  $L(n-K)^\alpha$; it is a monotone concave function satisfying
  $\Delta(n)\le\Delta^*(n)$ for all $n\le M$.
    
  Consider now $n>M$. We then have $\ell,m<n$. By induction, we have
  $\Delta(\ell)\le\Delta^*(\ell)$ and $\Delta(m)\le\Delta^*(m)$; so,
  using concavity of $\Delta^*$,
  \begin{align*}
    \Delta(n) &\le\Delta^*(\ell)+\Delta^*(m)\le 2\Delta^*\Big(\frac12(\ell+m)\Big)\le2\Delta^*\Big(\frac\eta2\big(n+\tfrac C\eta\big)\Big)\\
    &=2L\Big(\frac\eta2\big(n+\tfrac C\eta\big)-K\Big)^\alpha
    =2L\Big(\frac\eta2(n-K)\Big)^\alpha=L(n-K)^\alpha=\Delta^*(n).
  \end{align*}
  Therefore, $\Delta(n)\le\Delta^*(n)\sim(n-K)^\alpha$ for all $n\in\N$.
\end{proof}

\begin{proposition}\label{main}
  We have $\delta(w)\precsim\|w\|^\alpha$ for all $w\in
  \Omega$. Equivalently, $\Delta(n)\precsim n^\alpha$.
\end{proposition}
\begin{proof}
  We show that, for some $C$ and all $n\in\N$, there exists
  $\ell,m\in\N$ with $\ell+m \le \eta n+C$ and $\Delta(n) \le
  \Delta(\ell)+\Delta(m)$; the claimed upper bound on $\Delta$ will
  then follow by Lemma~\ref{lem:general}.
 
  Consider a word $w=w_1\dots w_n$ realizing the maximum in $\Delta$,
  assumed without loss of generality to be pre-reduced. We will study
  the inverted orbit of $w$ on $X$. Write as above
  $w=\varepsilon^s\pair<u,v>$ with $u=u_1\dots u_\ell$ and $v=v_1\dots
  v_m$; then, as we will see, the inverted orbit of $w$ is made of
  rays $\mathsf0x$ for $x$ in the inverted orbit of $u$, and of rays
  $\mathsf1x$ for $x$ in the inverted orbit of $v$. By
  Lemma~\ref{lem:contract}, we have $\ell+m\le\eta(n+\|a\|)$. A suffix
  $w'=w_{i+1}\dots w_n$ has the form
  $w'=\varepsilon^{s'}\pair<u',v'>$, in which $u',v'$ are respectively
  suffixes of $u,v$. We have $\rho w'=\mathsf1\rho v'$ if $s'=0$, and
  $\rho w'=\mathsf0\rho u'$ if $s'=1$. Therefore,
  \begin{align*}
    \Delta(n) &= \#\{\rho w_{i+1}\dots w_n\mid i=0,\dots,n\}\\
    &\le \#\{\mathsf0(\rho u_{j+1}\dots u_\ell)\mid j=0,\dots,\ell\}+\#\{\mathsf1(\rho v_j\dots v_m)\mid j=0,\dots,m\}\\
    &\le \Delta(\ell)+\Delta(m).\qedhere
  \end{align*}
\end{proof}

We now explore the range of $\delta(w)$, for various words
$w\in\{a,b,c,d\}^n$. Let us insist that different words $u,v\in
\Omega$ which have the same value in $\Grig$ may very well have widely
different $\delta$-values. The following result is not used in this
text, but is included to stress the difference between direct and
inverted orbit growth:
\begin{remark}\label{rem:invdir}
  \textit There exists for all $n$ a word $w\in \Omega$ of length
  $2^n$, whose direct orbit on $X$ has length $2^n$, and such that
  $\delta(w)\sim n$. In particular, $w$ has length $n$ in $\Grig$.
\end{remark}
\begin{proof}
  Consider again the substitution $\sigma:\Omega\to \Omega$ given
  by~\eqref{eq:sigma}, and the word $w_n=\sigma^{n-1}(ad)$. Let $v$ be
  a suffix of $w_n$.

  Note that $w_n=\pair<u,w_{n-1}>$ for some word $u$ over $\{a,d\}$;
  therefore, the suffix $v$ has the form $\varepsilon^s\pair<u',v'>$
  for some $s\in\{0,1\}$ and suffixes $u',v'$ of $u,w_{n-1}$
  respectively. Now either $i=1$, so $\rho v=\mathsf0(\rho
  u')\in\{\mathsf0\rho,\mathsf{00}\rho\}$, or $i=0$, and $\rho
  v=\mathsf1(\rho v')$; by induction, $\rho v$ can take at most $2n$
  values when $v$ ranges over all suffixes of $w_n$.

  On the other hand, $\rho w_n=\mathsf0^{n+1}\rho$ is at distance
  $2^n-1$ from $\rho$, so the direct orbit of $w_n$ traverses $2^n$
  points.
\end{proof}

\begin{remark}
  For comparison, let $\delta'(w)$ denote the size of the direct orbit
  of $w$; namely, $\delta'(w)=\#\{\rho,\rho w_1,\rho w_1w_2,\dots,\rho
  w_1\dots w_n\}$ if $w=w_1\dots w_n$. Then, for
  $w=\pair<u,v>\varepsilon^s$, we have the inequality
  $\delta'(w)\le2\delta'(v)$, from which nothing can be deduced,
  instead of $\delta(w)\le\delta(u)+\delta(v)$. As soon as $G$ is
  infinite, there is for all $n\in\N$ a word of length $n$ whose
  direct orbit visits $n$ points.
\end{remark}

\noindent We now show that the estimate in Proposition~\ref{main} is optimal:
\begin{proposition}\label{substitution}
  There exists a constant $C$ such that, for all $n\in\N$, there
  exists a word $w_n\in \Omega$ of length at most $C(2/\eta)^n$ with
  $\delta(w)\ge 2^n$.
\end{proposition}
\begin{proof}
  Write $\Omega'=\{ab,ac,ad\}^*\subset \Omega$, consider the substitution
  $\zeta:\Omega'\to \Omega'$ given by
  \[\zeta:ab\mapsto abadac,\quad ac\mapsto abab,\quad ad\mapsto acac,\]
  and consider the word $w_n=\zeta^n(ad)$. For example, we have
  $w_0=ad$, $w_1=acac$, $w_2=abababab$,
  $w_3=(abadac)^4$, \dots
  
  Counting the number of occurrences of $ab,ac,ad$ in the word $w_n$,
  we get
  \[|w_n|=\begin{pmatrix}2 & 2 & 2\end{pmatrix}\begin{pmatrix}1 &
    2 & 0\\ 1 & 0 & 2\\ 1 & 0 & 0\end{pmatrix}^n\begin{pmatrix}0\\ 0\\
    1\end{pmatrix};\] the characteristic polynomial of the $3\times3$
  matrix is $X^3-X^2-2X-4$. This polynomial has a positive real root
  $2/\eta$, and two conjugate complex roots of a smaller absolute
  value. Therefore, there exists a constant $C$ such that the length
  of $w_n$ is at most $C(2/\eta)^n$ for all $n$. (In fact, it is also
  $\ge C'(2/\eta)^n$ for another constant $C'$, because the polynomial
  is irreducible over $\mathbb Q$).

  Consider $w=\varepsilon^s\pair<u,v>$ and
  $w'=\varepsilon^{s'}\pair<u',v'>$. We write $w\bumpeq w'$ to mean
  that $s=s'$, that $u,u'$ have the same pre-reduction, and that
  $v,v'$ have the same pre-reduction.  Under the wreath recursion, we
  have
  \[\zeta(ab)=\varepsilon\pair<aba,c1d>\bumpeq\varepsilon\pair<aba,b>,\quad
  \zeta(ac)=\pair<ca,ac>,\quad \zeta(ad)=\pair<da,ad>.
  \]
  It follows that, for any word $av\in \Omega'$, we have
  \[\zeta(av)\bumpeq\begin{cases}\varepsilon\pair<ava,v> & \text{ if $\zeta(av)$ contains an odd number of $a$'s,}\\
    \pair<va,av> & \text{ if $\zeta(av)$ contains an even number of $a$'s.}
  \end{cases}
  \]
  In particular, let us denote by $w_{n-1}':=a^{-1}w_{n-1}a$ the word
  obtained from $w_{n-1}a$ by deleting its initial $a$; then
  $w_n\bumpeq\pair<w_{n-1}',w_{n-1}>$.

  Let now $u$ be a suffix of $w_{n-1}$. There exists then a suffix $v$
  of $w_n$ such that $v\bumpeq\pair<*,u>$. Similarly, let $u'$ be a
  suffix of $w_{n-1}'$. There exists then a suffix $v'$ of $w_n$ such
  that $v'\bumpeq\varepsilon\pair<u',*>$. Now $\rho v=\mathsf1\rho u$,
  and $\rho v'=\mathsf0\rho u'$; so $\mathcal
  O(w_n)\supseteq\mathsf1\mathcal O(w_{n-1})\sqcup\mathsf0\mathcal
  O(w_{n-1}')$, and
  \[\delta(w_n)\ge\delta(w_{n-1})+\delta(w_{n-1}').\]
  A similar reasoning shows $\delta(w_n')\ge2\delta(w_{n-1}')$. Since
  $\delta(w_0)=2$ and $\delta(w_0')=1$, we get
  \[\delta(w_n)\ge2^n+1,\qquad\delta(w_n')\ge2^n.\]
  (These inequalities are in fact equalities).
\end{proof}

\begin{corollary}
  There exist constants $C_1,C_2>0$ such that, for any $\ell\in\N$,
  \begin{enumerate}
  \item for all words $w$ of length $\ell$, we have $\delta(w)\le C_1\ell^\alpha$;
  \item there exists a word $w$ of length $\ell$ with $\delta(w)\ge C_2\ell^\alpha$.
  \end{enumerate}
\end{corollary}

We will also need to control the total number of possibilities
$\Sigma(n)$ for the inverted orbit of a word of length $\le n$. Since
the inverted orbit has cardinality at most $\sim n^\alpha$, and lies
in the ball of radius $n$ in the Schreier graph $\rho G$, which
contains at most $2n+1$ points (see Figure~\ref{fig:schreier}), we
have $\Sigma(n)\precsim\binom{2n+1}{n^\alpha}$; but this estimate is
too crude for our purposes. We improve it as follows:
\begin{lemma}\label{support}
  Set $\Sigma(n)=\#\{\mathcal O(w)\mid n\ge|w|\}$. Then
  $\Sigma(n)\precsim\exp(n^\alpha)$.
\end{lemma}
\begin{proof}
  Recalling that the norms $\|\cdot\|$ and $|\cdot|$ are equivalent,
  we consider $w$ with $\|w\|\le n$, assumed without loss of
  generality to be pre-reduced (because the inverted orbit is invariant
  under pre-reduction). We write $w=\varepsilon^s\pair<u,v>$ with
  $\|u\|=\ell,\|v\|=m$, and recall that a suffix $w'$ of $w$ has the
  form $w'=\varepsilon^{s'}\pair<u',v'>$ for suffixes $u',v'$ of $u,v$
  respectively. As in Proposition~\ref{main}, we then have $\mathcal
  O(w)=\mathcal O(u)\sqcup\mathcal O(v)\varepsilon$; we get
  \[\Sigma(n)\le\sum_{\ell+m\le\eta(n+\|a\|)}\Sigma(\ell)\Sigma(m).\]
  We reuse the notation of Proposition~\ref{main}, and take as Ansatz
  \begin{equation}\label{eq:Ansatz}
    \Sigma(n)\le\exp(\Delta^*(n))\eta/4n,
  \end{equation}
  for a function $\Delta^*$ as in~\eqref{eq:Delta}, with a large
  enough constant $L$ that~\eqref{eq:Ansatz} holds whenever $n\le
  M$. Then, because $\exp(\Delta^*(n))/n$ is log-concave, we get as
  before
  \[\Sigma(n)\le\eta n(\exp(K(n\eta/2)^\alpha)/(2n))^2=\exp(Kn^\alpha)\eta/4n\]
  and~\eqref{eq:Ansatz} holds by induction for all $n\in\N$.
\end{proof}

\section{Groups of intermediate growth}
Recall that $\Grig$ denotes the first Grigorchuk group, and that $X$
denotes the orbit of $\rho=\mathsf1^\infty$ under the right
$\Grig$-action on the (boundary of the) binary tree $\mathcal
T=\{\mathsf0,\mathsf1\}^*$.

It is convenient, when considering groups of intermediate growth, to
write their growth function in the form $\exp(n/\phi(n))$, for an
unbounded function $\phi$.

\begin{lemma}\label{lem:interm}
  Let $A$ be a non-trivial group having growth $\sim\exp(n/\phi(n))$,
  and assume that $n/\phi(n)$ is concave. Consider the wreath product
  $W=A\wr_X\Grig$. Then the growth of $W$ is $\sim
  \exp(n/\phi(n^{1-\alpha}))$.
\end{lemma}

\begin{proof}
  We begin by the lower bound. For $n\in\N$, consider a word $w$ of
  length $n$ with $\delta(w)\sim n^\alpha$, which exists by
  Proposition~\ref{substitution}; write $\mathcal O(w)=\{x_1,\dots,x_k\}$
  for $k\sim n^\alpha$. Choose then $k$ elements $a_1,\dots,a_k$ of
  length $\le n^{1-\alpha}$ in $A$. Define $f\in \sum_XA$ by
  $f(x_i)=a_i$, all unspecified values being $1$. Then $wf\in W$ may
  be expressed as a word of length $n+|a_1|+\dots+|a_k|\sim 2n$ in the
  standard generators of $W$.

  Furthermore, different choices of $a_i$ yield different elements of
  $W$; and there are $\sim
  \exp(n^{1-\alpha}/\phi(n^{1-\alpha}))^{n^\alpha}=\exp(n/\phi(n^{1-\alpha}))$
  choices for all the elements of $A$. This proves the lower bound.

  For the upper bound, consider a word $w$ of length $n$ in $W$, and
  let $f\in\sum_XA$ denote its value in the base of the wreath
  product. The support of $f$ has cardinality at most
  $\delta(w)\precsim n^\alpha$ by Proposition~\ref{main}, and may take
  at most $\sim\exp(n^\alpha)$ values by Lemma~\ref{support}.

  Write then $\sup(f)=\{x_1,\dots,x_k\}$ for some $k\precsim
  n^\alpha$, and let $a_1,\dots,a_k\in A$ be the values of $w$ at its
  support; write $\ell_i=\|a_i\|$. We now consider two cases. If $A$
  is finite, then each of the $a_i$ may be chosen among, at most,
  $\#A$ possibilities, so there are $\sim\exp(n^\alpha)$ possibilities
  in total for the element $f$.

  Assume now that $A$ is infinite, so that $v_A(n)\succsim n$. Since
  $\sum\ell_i\le n$, the lengths of the different elements on the
  support of $f$ define a composition of a number not greater than $n$
  into at most $n^\alpha$ summands; such a composition is determined
  by $n^\alpha$ ``marked positions'' among $n+n^\alpha$, so there are
  at most $\binom{n}{n^\alpha}\sim\exp(\log(n)n^\alpha)$ such
  compositions. Each of the $a_i$ is then chosen among $v_A(\ell_i)$
  elements, and there are $\sim\exp(\ell_i/\phi(\ell_i))$ such choices
  for each $i$.

  By our concavity assumption, there are at most
  $\sim\prod\exp(\ell_i/\phi(\ell_i))\precsim\exp(n/\phi(n^{1-\alpha}))$
  choices for the elements in $A$.

  We have now decomposed $w$ into data that specify it uniquely, and
  we multiply the different possibilities for each of the pieces of
  data. First, there are $\precsim\exp(n^\alpha)$ possibilities for
  the value of $w$ in $\Grig$, by the upper
  bound~\eqref{eq:grigupper}. There are $\precsim\exp(n^\alpha)$
  possibilities for the support of $w$. There are
  $\precsim\exp(\log(n)n^\alpha)\exp(n/\phi(n^{1-\alpha}))$
  possibilities for the values of $w$ at its support, the first factor
  counting the number of compositions of $n$ as a sum of $n^\alpha$
  terms and the second factor counting the number of elements in $A$
  of these lengths. Altogether, we get
  \[v_W(n)\precsim\exp(n^\alpha+n^\alpha+n^\alpha\log(n)+n/\phi(n^{1-\alpha}))\sim\exp(n/\phi(n^{1-\alpha})),
  \]
  and we have obtained the claimed upper bound.
\end{proof}

We are ready to prove the first part of Theorem~\ref{mainthm:A}.
\begin{theorem}\label{thm:K}
  Consider the following sequence of groups: $K_0=\Z/2\Z$, and
  $K_{k+1}=K_k\wr_X\Grig$. Then every $K_k$ is a finitely generated
  infinite torsion group, with growth function
  \[
  v_{K_k}(n)\sim \exp(n^{1-(1-\alpha)^k}).
  \]
\end{theorem}
\begin{proof}
  We start by $\phi_0=n$; then
  $\phi_{k+1}(n)=\phi_k(n^{1-\alpha})$, so
  $\phi_k(n)=n^{(1-\alpha)^k}$.
\end{proof}

\begin{theorem}\label{thm:L}
  Consider the following sequence of groups: $L_0=\Z$, and
  $L_{k+1}=L_k\wr_X\Grig$. Then their growth functions satisfy
  \[
  v_{L_k}(n)\sim\exp(\log(n)n^{1-(1-\alpha)^k}).
  \]
\end{theorem}
\begin{proof}
  We start by $\phi_0=n/\log(n)$; then
  $\phi_{k+1}(n)=\phi_k(n^{1-\alpha})$. Now
  $\log(n^{1-\alpha})\sim\log(n)$, so we get $\phi_k(n)\sim
  n^{(1-\alpha)^k}/\log(n)$.
\end{proof}

\begin{example}\label{ex:expgrowth}
  The inverted orbits of $\Erschler$, for its action on the orbit $X$
  of the rightmost ray $\rho=\mathsf1^\infty$, have linear growth.
  Moreover, because $\Erschler$ contains the infinite-order element
  $ad$ acting freely on $X$, as we saw in~\eqref{eq:freediadicorbit},
  the permutational wreath product of any group $A$ with
  $(\Erschler,X)$ contains the wreath product $A\wr\Z$. Therefore, for
  any $A \ne \{1\}$ the growth of this permutational wreath product is
  exponential.
\end{example}

\begin{example}\label{ex:differentgrowth}
  Consider $G= G'=\Grig \times \Erschler$, and let $X$ and $X'$ be
  respectively the orbits of the rightmost ray $\rho=\mathsf1^\infty$
  under the action of $\Grig$, respectively $\Erschler$, on the rooted
  tree.  Extend the action of $\Grig$ on $X$ to an action of $G$ by
  making $\Erschler$ act trivially, and extend the action of
  $\Erschler$ on $X'$ to an action of $G'$ by making $\Grig$ act
  trivially.

  Let $A$ be a finite group containing at least two elements. The
  unmarked Schreier graph of $(G,X)$ is the same as the unmarked
  Schreier graph of $(G',X')$.
  
  However, the growth of the wreath product $A\wr_X G$ is
  subexponential, whereas the growth of $A\wr_{X'}G'$ is exponential.
\end{example}

\begin{figure}
  \begin{center}
  \begin{tikzpicture}[xscale=0.7]
  \footnotesize
  \path (0,0) edge[loop above] node {$b$} ();
  \path (0,0) edge[loop left] node {$c$} ();
  \path (0,0) edge[loop below] node {$d$} ();
  \foreach\i in {0,2,...,14} \path (\i,0) edge node[above] {$a$} (\i+1,0);
  \foreach\i/\g in {1/b,3/d,5/b,7/c,9/b,11/d,13/b,15/b} \path (\i,0) edge[bend left] node[above] {$\g$} (\i+1,0);
  \foreach\i/\g in {1/c,3/b,5/c,7/d,9/c,11/b,13/c,15/c} \path (\i,0) edge[bend right] node[below] {$\g$} (\i+1,0);
  \foreach\i/\g in {1/d,2/d,3/c,4/c,5/d,6/d,7/b,8/b,9/d,10/d,11/c,12/c,13/d,14/d,15/d,16/d} \path (\i,0) edge[loop above] node[above] {$\g$} ();
  \path (16,0) edge[densely dotted] (17,0);
  \end{tikzpicture}
  \end{center}
  \caption{The Schreier graph of $\Grig$. The leftmost point is the ray $\rho=\mathsf{111}\dots$, followed by $\mathsf0\rho,\mathsf{00}\rho,\mathsf{10}\rho,\dots$}\label{fig:schreier}
\end{figure}

Indeed, observe that $W=A\wr_XG = A\wr_X \Grig \times \Erschler$. By
Theorem~\ref{thm:K} we know that $K_1= A\wr_X \Grig$ has intermediate
growth. We see that $W$ is a direct sum of two groups of intermediate
growth, and hence the growth of this group is intermediate.

On the other hand $W'=A\wr_{X'}G' = A\wr_{X'}\Erschler \times \Grig$,
and already the first factor has exponential growth, see
Example~\ref{ex:expgrowth}.

The unmarked Schreier graph of $(\Grig,X)$, as well as the Schreier
graph of $(\Erschler,X')$, are rays, in which every second edge edge
has been duplicated, a loop has been added at each vertex, and three
loops are drawn at the origin (see Figure~\ref{fig:schreier}). The
unmarked Schreier graphs of $(G,X)$ and $(G',X')$ are obtained from
that graph by drawing four additional loops at each vertex.

\begin{remark}\label{rem:nilpotentactions}
  \textit Let $N$ be a finitely generated nilpotent group, acting
  transitively on an infinite set $X$. Then the inverted orbits for
  this action have linear growth: that is, there exists $C>0$ such
  that for any $n>1$ there exist a word $w_n$ of length $n$ such that
  its inverted orbit for the action on $X$ visits at least $Cn$
  points.
\end{remark}
\begin{proof}
  Take $A=\Z/2\Z$ and let $G$ be the wreath product of $A$ with
  $(N,X)$. Observe that $G$ is an extension of an Abelian group by a
  nilpotent group, so $G$ is solvable. Since $N$ and $A$ are finitely
  generated, so is $G$.  We know that $G$ contains as a subgroup
  $\sum_X A$. Since $G$ contains an infinitely generated subgroup, we
  conclude that $G$ is not virtually nilpotent.  Therefore, $G$ has
  exponential growth. However, $N$ has subexponential growth, so if it
  also had sublinear inverted orbit growth then $G$ would have
  subexponential growth.
\end{proof}

\subsection{Torsion-free examples}\label{ss:tf}
Grigorchuk constructed in~\cite{grigorchuk:pgps}*{\S5} a torsion-free
group $H$ of intermediate growth. We recall the basic steps: start by
\[H_0=\langle a,b,c,d\mid
[a^2,b],[a^2,c],[a^2,d],[b,c],[b,d],[c,d]\rangle,\] and define a
wreath recursion $\psi:H_0\to H_0\wr \mathfrak S_2$ by the same
formula~\eqref{eqgrig} as for Grigorchuk's first group, namely
\[
\psi:a\mapsto\pair<1,1>\varepsilon,\quad b\mapsto\pair<a,c>,\quad
c\mapsto\pair<a,d>,\quad d\mapsto\pair<1,b>.
\]
Set $K_0=1$ and inductively $K_{n+1}=\psi^{-1}(K_n\times K_n)$. Define
then $H=H_0/\bigcup_{n\ge0}K_n$, and use the same notation for the
generators $a,b,c,d$ of $H$ and the induced homomorphism $\psi:H\to
H\wr\mathfrak S_2$. Note that $H$ is the largest quotient of $H_0$
such that the restriction of $\psi$ to $\langle
b,c,d,b^a,c^a,d^a\rangle$ is injective.

We view the group $H$ in terms of permutational extensions, and
compute its growth function by adapting Lemma~\ref{lem:interm}.

Note first that the natural map $\xi:a\mapsto a,b\mapsto b,c\mapsto
c,d\mapsto d$ defines a homomorphism from $H$ to $\Grig$. Consider the
subgroup $C=\langle b^2,c^2,d^2,bcd\rangle$ of $H$; then, for instance
using the presentation~\eqref{eq:lysionok} of $\Grig$, the normal closure
$\langle a^2,C\rangle^H$ equals $\ker\xi$.

Grigorchuk proves in~\cite{grigorchuk:pgps}*{pages~199--200} that
$C\cong\Z^3$, and that $\langle a^2\rangle\cong\Z$.  Because $a^2$ is
central in $H$, we have exact sequences
\begin{gather*}
  1\longrightarrow \langle a^2\rangle\longrightarrow H\longrightarrow H/\langle a^2\rangle\longrightarrow 1,\\
  1\longrightarrow C^H\longrightarrow H/\langle a^2\rangle\longrightarrow\Grig\longrightarrow 1.
\end{gather*}

Let $\psi_0,\psi_1$ be the co\"ordinates of $\psi$, namely, the
set-maps defined by the projections $H\to H\wr\mathfrak S_2\to H\times
H\to H$, as in $\psi(g)=\pair<\psi_0(g),\psi_1(g)>\varepsilon^s$. Note
that $\psi_0,\psi_1$ are not homomorphisms, but their restriction to
$B:=\langle b,c,d\rangle$ is a homomorphism; $\psi_0$ maps to $\langle
a\rangle$, while $\psi_1$ permutes cyclically $b,c,d$. Consider
$\tau=\tau_1\tau_2\dots\in\{\mathsf0,\mathsf1\}^\infty$ a ray in
$\mathcal T$. Given $g\in H$, it is easy to see
(see~\cite{grigorchuk:pgps}*{page~200}) that
$\psi_{\tau_n}\psi_{\tau_{n-1}}\dots\psi_{\tau_1}(g)$ belongs to
$B:=\langle b,c,d\rangle$ for all $n$ large enough. Recall the
endomorphism $\sigma$ from~\eqref{eq:sigma}; it induces an
automorphism of $B$ permuting cyclically $b,d,c$, so we have
$\sigma(\psi_1(g))=g$ for all $g\in B$. The sequence
$\sigma^n\psi_{\tau_n}\psi_{\tau_{n-1}}\dots\psi_{\tau_1}(g)$
eventually stabilizes, and we call its limit $g_\tau\in B$ the
\emph{germ} of $g$ at $\tau$. Note that $g_\tau=1$ unless $\tau$ is in
the $H$-orbit of $\rho=\mathsf1^\infty$. For an element $g\in B$, its
germs are $g_\rho=g$ and $g_\tau=1$ for all $\tau\neq\rho$. Similarly,
for $g\in B$ and $x\in H$, the germs of the conjugate $x^{-1}gx=g^x\in
B^x$ are $(g^x)_{\rho x}=g$ and $(g^x)_\tau=1$ for all $\tau\neq\rho
x$.

\begin{lemma}\label{lem:kerxi}
  An element of $H$ is determined by its projection to $\Grig$, its
  $a$-exponent sum, and its germs:
  \[\ker\xi=\langle a^2\rangle\times\sum_{x\in
    (\Grig)_\rho\backslash\Grig}C^x\cong\Z\times\sum_X\Z^3.\]
\end{lemma}
\begin{proof}
  On the one hand, $a^2$ is central, and generates a split copy of
  $\Z$. As we noted above, $\ker\xi$ is generated by $a^2$ and
  conjugates of $C$.

  Consider next $y,z\in C$, and $g\in H$; we show that $y$ and $z^g$
  commute. Write $g=\pair<g',g''>\varepsilon^s$ for $g',g''\in H$ and
  $s\in\{0,1\}$. Write also $y=\pair<a^{2k},y'>$ and
  $z=\pair<a^{2\ell},z'>$.  If $s=1$, then
  $z^g=\pair<(z')^{g'},(a^{2\ell})^{g''}>=\pair<(z')^{g'},a^{2\ell}>$
  using the relations in $H_0$; so
  $[y,z^g]=\pair<[(z')^{g'},a^{2k}],[a^{2\ell},y']>=1$. If $s=0$, then
  $z^g=\pair<(a^{2\ell})^{g'},(z')^{g''}>=\pair<a^{2\ell},(z')^{g''}>$;
  so
  $[y,z^g]=\pair<[a^{2\ell},a^{2k}],[y',(z')^{g''}]>=\pair<1,[y',(z')^{g''}]>$,
  and now, because $g''$ is shorter than $g$, we eventually have $g\in
  B$, so by induction $[y,z^g]=1$.

  Consider finally $y\in C$ and $h\in H$ whose image $\xi(h)$ in
  $\Grig$ fixes $\rho$; we show that $y$ and $h$ commute. Write again
  $y=\pair<a^{2k},y'>$ and $h=\pair<h',h''>$; then
  $[y,h]=\pair<1,[y',h'']>$, and by induction eventually $h\in B$ so
  $[y,h]=1$. It now follows that $\ker\xi$ is a quotient of
  $\Z\times\sum_X\Z^3$.

  On the other hand, consider an element $h$ of $H$ of the form
  $y_1^{g_1}\dots y_\ell^{g_\ell}$ for some distinct $g_i\in
  (\Grig)_\rho\backslash\Grig$ and $y_i\in C$. Consider some
  $i\in\{1,\dots,\ell\}$; then the germ $h_{\rho g_i}$ equals $y_i$,
  so no relations occur among the elements of $\Z\times\sum_X\Z^3$
  when they are mapped to $\ker\xi$.
\end{proof}

The difference between $H$ and the wreath product $\Z^3\wr_X\Grig$ is
twofold: $H$ is not a split extension of $\sum_X\Z^3$; and the
generator $a\in\Grig$ lifts to an infinite-order element of $H$. We
nevertheless show that $H$ and $\Z^3\wr_X\Grig$ have the same
asymptotic growth:
\begin{proposition}\label{prop:tf}
  The group $H$ has growth $\sim\exp(\log(n)n^\alpha)$.
\end{proposition}
\begin{proof}
  We define a set-theoretic splitting $\nu$ of $\xi:H\to\Grig$ by
  the condition that, for all $g\in\Grig$, the germs $\nu(g)_\tau$
  all belong to $\{1,b,c,d\}$, and that the total exponent sum $|\nu(g)|_a$
  of $a$ in $\nu(g)$ is $0$ or $1$.

  By Lemma~\ref{lem:kerxi}, elements $h\in H$ can, and will, be put in
  the form $a^{2\ell}f\nu(g)$, with $g\in\Grig$, $f:X\to C$ finitely
  supported, and $\ell\in\Z$. We consider the effect of
  left-multiplying such an expression by a generator
  $t\in\{a,b,c,d\}^{\pm1}$. First, consider $t=a^k$ for
  $k\in\{1,-1\}$. Write $|\nu(g)|_a+k=n+2m$, with
  $n\in\{0,1\}$. Then
  \[th = a^{2\ell}tf\nu(g)=a^{2(\ell+m)}(t\cdot f)\nu(ag).
  \]
  Consider next $t\in\{b,c,d\}^{\pm1}$. Then
  \[th = a^{2\ell}tf\nu(g)=a^{2\ell}(t\cdot f)t\nu(g);
  \]
  now, in $\langle b,c,d\rangle$, write $t\nu(g)_\rho=zr$ for $z\in
  C$ and $r\in\{1,b,c,d\}$. Denote still by $z$ the function $X\to C$
  which takes value $z$ at $\rho$ and is trivial everywhere else. We
  then have
  \[th = a^{2\ell}(t\cdot f)z\nu(tg).
  \]
  It follows that the action of a generator on an element of $H$,
  written in the form $a^{2\ell}f\nu(g)$, is by translation of $f$
  (just as in the wreath product $\Z^3\wr_X\Grig$), possibly followed
  by a multiplication at $\rho$ of $f$ by a generator of $C$ or its
  inverse.

  More pedantically, the computation above shows that the cocycle
  $\eta(g,h):=\nu(gh)^{-1}\nu(g)\nu(h)$ associated with the
  extension $1\to C^H\to H\to\Grig\to 1$ is controlled in the
  following manner: if $|g|,|h|\le n$, then $\eta(g,h):X\to C$ is
  supported on a set of cardinality $n^\alpha$, and takes values
  bounded in $\{-n,\dots,n\}$.

  The remainder of the growth computation follows closely the argument
  in Lemma~\ref{lem:interm}; we repeat it briefly.

  Consider the representations as $a^{2\ell}f\nu(g)$ of elements $h\in
  H$ of norm at most $n$. The element $g=\xi(h)$ belongs to $\Grig$
  and has norm at most $n$, so may take at most $\exp(Dn^\alpha)$
  values, for a predefined constant $D$. The function $f$ is supported
  on a set of cardinality at most $Cn^\alpha$, for another predefined
  constant $C$, and takes values in $\{-n,\dots,n\}$; so there are at
  most $\exp(\log(2n+1)Cn^\alpha)$ possibilities for $f$. Finally
  $|\ell|\le n$. In total, there are $\precsim\exp(\log(n)n^\alpha)$
  values for $h$.

  For the lower bound, consider a word $w=w_1\dots w_n$ of length $n$
  over $\{a,b,c,d\}$ with $\delta(w)\sim n^\alpha$, which exists by
  Lemma~\ref{substitution}; write $\mathcal O(w)=\{x_1,\dots,x_k\}$
  for $k\sim n^\alpha$, and let $i_1,\dots,i_k\in\{1,\dots,n\}$ be
  such that $\rho w_{i_j}\dots w_n=x_j$ for $j=1,\dots,k$.  Choose
  then $k$ numbers $a_1,\dots,a_k$ in $\Z\cap[1,n^{1-\alpha}]$. Insert
  $(bcd)^{a_j}$ before position $i_j$ in $w$, and call the resulting
  word $v(a_1,\dots,a_k)$.

  First, the length of $v(a_1,\dots,a_k)$ is at most $n+3n^\alpha
  n^{1-\alpha}=4n$. Then, the germ at $x_j$ of $v(a_1,\dots,a_k)$
  belongs to $\{1,b,c,d\}(bcd)^{a_j}$, so all $v(a_1,\dots,a_k)$ are
  distinct. It follows that there are at least
  $(n^{1-\alpha})^{n^\alpha}$ elements of length $4n$ in $H$.
\end{proof}

\begin{proof}[Proof of Theorem~\ref{mainthm:A}, second part]
  For $k=0$, consider the group $H_0=\Z$; for $k=1$, consider the
  group $H_1=H$ from Proposition~\ref{prop:tf}. For $k>1$, consider
  inductively $H_k=H_{k-1}\wr_X H$. They are torsion-free, as
  extensions of a torsion-free group by a torsion-free group.
\end{proof}

\subsection{Orbits on pairs of rays}
We gather here some results
from~\cite{bartholdi-g:parabolic}. Consider the ray
$\rho=\mathsf1^\infty$ the ray in the binary tree $\mathcal T$, and
its orbit $X:=\rho\Grig$. The group $\Grig$ acts on $X$, and therefore
acts (diagonally) on $X\times X$.

Because $\Grig$ acts transitively on $X$, the $\Grig$-orbits on
$X\times X$ are in bijection with the orbits of the stabilizer
$P=(\Grig)_\rho$ on $X$, and also in bijection with the double cosets
$PgP$ of $P$ in $\Grig$.

The set of orbits of $\Grig$ on $X\times X\setminus\{(x,x)\mid x\in
X\}$ may be readily described. A pair of distinct points $(x,y)\in
X\times X$ determines a bi-infinite path in $\mathcal T$, namely the
path $\gamma$ that starts from $x$, goes to the root of $\mathcal T$,
and leaves towards $y$. Let $\overline\gamma$ denote the geodesic path
(without backtracking) associated with $\gamma$, and let $(x|y)\in\N$
denote the minimal distance of this geodesic to the root of $\mathcal
T$.  The action of $\Grig$ on $\mathcal T$ and on $\partial\mathcal T$
induces an action on bi-infinite geodesics in $\mathcal T$, so $(x|y)$
is preserved by the $\Grig$-action. We now show that pairs $(x,y)$ and
$(x',y')$ belong to the same $\Grig$-orbit if and only if
$(x|y)=(x'|y')$.

\noindent We summarize the results:
\begin{lemma}[\cite{bartholdi-g:parabolic}*{Lemma~9.10}]\label{lem:orbits2}
  The orbits of the stabilizer $P=(\Grig)_\rho$ of $\rho$ on
  $X\setminus\{\rho\}$ are described as follows:
  \begin{equation}\label{eq:orbits}
    \mathcal O=\{\mathsf1^k\mathsf0\{\mathsf0,\mathsf1\}^*\rho\mid k\in\N\}.
  \end{equation}
  In particular, there are infinitely many orbits of $\Grig$ on
  $X\times X$.  For $x\neq y\in X$, we denote by $(x|y)\in\N$ the
  length of the maximal common prefix of $x,y$; it is the distance to
  the root of the geodesic in $\mathcal T$. We also set
  $(x|x)=\infty$.  The orbit of $(x,y)$ is then completely determined
  by $(x|y)\in\N\cup\{\infty\}$.

  Recall the endomorphism $\sigma$ from~\eqref{eq:sigma}. The set
  $T=\{\sigma^n(a)\mid n\in\N\}$ is a set of non-trivial double coset
  representatives of $P$.\qed
\end{lemma}

\noindent A generating set for $P$ has also been computed:
\begin{lemma}[\cite{bartholdi-g:parabolic}*{Theorem~4.4}]\label{lem:genP}
  The stabilizer $P=(\Grig)_\rho$ is generated by
  \[U=\bigcup_{n\in\N}\sigma^n\{b,c,d,d^a,(ac)^4,\sigma([a,b])^a,\sigma^2([a,b])^a\}.\]
\end{lemma}

\subsection{Presentations for wreath products}
We recall the notion of $L$-presentation, introduced
in~\cite{bartholdi:lpres}. A group $G$ is \emph{finitely
  $L$-presented} if there exists a finitely generated free group
$F=\langle S\rangle$, a finite set $\Phi$ of endomorphisms of $F$, and
finite subsets $Q,R$ of $F$, such that $G\cong F/\langle
Q\cup\bigcup_{\phi\in\Phi^*}\phi(R)\rangle^F$. The expression $\langle
S|Q|\Phi|R\rangle$ is the corresponding \emph{finite $L$-presentation}.

In particular, the first Grigorchuk group is finitely $L$-presented as
\[\Grig=\langle a,b,c,d||\sigma|a^2,bcd,[d,d^a],[d,d^{(ac)^2a}]\rangle.\]

\begin{proposition}
  Let $A$ be a finitely $L$-presented group. Then $A\wr_X\Grig$ is
  finitely $L$-presented.
\end{proposition}
\begin{proof}
  Cornulier characterizes in~\cite{cornulier:fpwreath} when
  permutational wreath products are finitely presented. A
  permutational wreath product $A\wr_X G$, for $G$ acting transitively
  on a transitive $G$-set $X=P\backslash G$, is presented as follows:
  as generators, take those of $A$ and $G$. As relations, take: those
  of $A$ and $G$; the relation $[a,u]$ for every generator $a$ of $A$
  and every $u$ in a generating set $U$ of $P$; and the relations
  $[a,b^t]$ for every generators $a,b$ of $A$ and $q$ in a set of
  double coset representatives of $P$ in $G$; namely $t\in T$ with
  $G=P\sqcup\bigsqcup_{g\in T} PgP$.

  In the case of the first Grigorchuk group, a generating set for
  $P=(\Grig)_\rho$, and a set of double coset representatives, have
  been computed in~\cite{bartholdi-g:parabolic}; see
  Lemmata~\ref{lem:orbits2} and~\ref{lem:genP}.  Let $\langle
  S|Q|\Phi|R\rangle$ be a finite $L$-presentation of $A$. A finite
  $L$-presentation for $A\wr_X\Grig$ is then $\langle
  S,a,b,c,d|Q|\Phi'\cup\{\sigma'\}|R\cup R'\rangle$, with
  \begin{align*}
    \Phi' & \text{ the endomorphisms in $\Phi$, extended by fixing $a,b,c,d$;}\\
    \sigma' &= \sigma \text{ on }\{a,b,c,d\},\text{ and fixing }S;\\
    R' &= \{[s,b],[s,d^a],[s,(ac)^4],[s,\sigma([a,b])^a],[s,\sigma^2([a,b])^a]\mid s\in S\}\\
    & \qquad{}\cup\{[s',s^a]\mid s,s'\in S\}.\qedhere
  \end{align*}
\end{proof}

\begin{corollary}
  The groups $K_k$ from Theorem~\ref{thm:K} and $L_k$ from
  Theorem~\ref{thm:L} are finitely L-presented.\qed
\end{corollary}

\begin{example}
  A recursive presentation for the group $K_1=\Z/2\Z\wr_X\Grig$ is
  \[K_1=\langle a,b,c,d,s\mid a^2,b^2,c^2,d^2,s^2,bcd,
  \sigma^n(r_1),\dots,\sigma^n(r_8)\text{ for all }n\in\N\rangle
  \]
  for $\sigma$ the same endomorphism as in~\eqref{eq:sigma}, extended by
  $\sigma(s)=s$, and iterated relations
  \begin{xalignat*}{2}
    r_1&=[d,d^a], &
    r_2&=[d,d^{(ac)^2a}],\\
    r_3&=[s,s^a], &
    r_4&=[s,b],\\
    r_5&=[s,d^a], &
    r_6&=[s,(ac)^4],\\
    r_7&=[s,\sigma([a,b])^a], &
    r_8&=[s,\sigma^2([a,b])^a].
  \end{xalignat*}
\end{example}

\section{Embeddings of the group of finitely supported permutations}


\begin{theorem}\label{thm:embedSinfty}
  There exists a group $H$ of intermediate growth that contains as a
  subgroup the group $\mathfrak S_\infty$ of finitely supported
  permutations of an infinite countable set.

  Moreover, the group $H$ can be chosen in such a way that its growth
  function satisfies
  \[\exp(n^\alpha)\precsim v_H(n)\precsim\exp(\log(n)n^\alpha).\]
\end{theorem}
\begin{proof}
  If $G$ acts on a set $X$, then it acts on the group of finitely
  supported permutations of $X$: for $\sigma: X \to X$ with
  $\sigma(x)=x$ for all $x\in X$ except finitely many, $g^{-1}\sigma
  g$ is still finitely supported for all $g\in G$.

  Let $X$ denote the orbit of the ray $\rho=\mathsf1^\infty$ under the
  action of the first Grigorchuk group $\Grig$, and let $\mathfrak
  S_\infty$ denote the group of finitely supported permutations of
  $X$. Set $H=\mathfrak S_\infty\rtimes\Grig$.  Take as generating set
  for $H$ the union of the generating set $\{a,b,c,d\}$ of $G$ with
  the involution $s$ switching $\mathsf0\rho$ and $\rho$.

  Consider an element $g=s^{e_1}h_1\dots s^{e_n}\dots h_n$ of length
  $\le2n$ in this group, for some $h_i\in\{a,b,c,d\}$ and
  $e_i\in\{0,1\}$, and write $g_i=h_ih_{i+1}\dots h_n$.  Observe that
  the finitely supported permutation corresponding to this element is obtained
  as a product of some of the involutions switching
  $(\mathsf0\rho)g_i$ and $\rho g_i$. By Proposition~\ref{main}, there
  are $\precsim n^\alpha$ possibilities for the first, and in view of
  Lemma \ref{le:indepbasepoint} also for the second of these two
  points;
  so the support of the permutation has
  cardinality $\precsim n^\alpha$. An element of $H$ of length
  $\sim n$ may therefore be described by: an element of $\Grig$ of
  length $\precsim n$; a subset of $X$ of cardinality $\precsim
  n^\alpha$; and a permutation of that subset. There are at most,
  respectively, $v_\Grig(n)$, $\binom{n}{n^\alpha}$, and $(n^\alpha)!$
  choices for each of these pieces of data.  We get
  \[v_H(n) \precsim \exp(n^\alpha) \exp(\log(n)n^\alpha)(n^\alpha)!\sim\exp(\log(n)n^\alpha).
  \]
  On the other hand, consider the word $w_n=g_1\dots g_\ell$ given by
  Proposition~\ref{substitution}. Set inductively $S_0=\emptyset$ and
  $S_k=S_{k-1}\cup\{\rho g_k\dots g\ell,\mathsf0\rho g_k\dots
  g_\ell\}$; and select positions $k_1,k_2,\dots,k_{2^n}$ such that
  $S_k\neq S_{k-1}$. Consider then the $2^{2^n}$ elements of $H$
  obtained by inserting, at each position $k_i$ in $w_n$, the word
  $s^{e_i}$ for all choices of $e_i\in\{0,1\}$.

  An easy induction shows that all these elements are distinct in $H$:
  given such an element, expressed as $\sigma g$ with
  $\sigma\in\mathfrak S_\infty$ and $g=w_n\in\Grig$, we recover the
  $e_i$ as follows. Let $x\in X$ be in the support of $\sigma$, and in
  $S_{k_i}\setminus S_{k_i-1}$ for maximal $k_i$. This determines
  $e_i=1$. Right-divide then by $s^{e_i}g_{k_i+1}\dots g_\ell$, and
  proceed inductively. All other $e_j$'s are $0$.

  This gives $v_H(2|w_n|)\ge2^{2^n}$, proving the lower bound.
\end{proof}

The first examples of groups of intermediate growth that are not
residually finite are constructed
in~\cite{erschler:nonrf}. Theorem~\ref{thm:embedSinfty} gives new
examples of this kind:
\begin{corollary}
  There exist finitely generated groups of growth
  $\precsim\exp(\log(n)n^\alpha)$ that contain an infinite simple
  group as a (normal) subgroup. In particular, such groups provide new
  examples of non residually finite groups of intermediate growth.
\end{corollary}
\begin{proof}
  We recall that the group of even permutations is a characteristic
  subgroup of the group $\mathfrak S_\infty$ of finitely supported
  permutations, and is simple.
\end{proof}

\begin{remark}
  It is shown in~\cite{erschler:boundarysubexp} that a there exist
  groups of intermediate growth which admit a non-degenerate measure
  with non-trivial Poisson-Furstenberg boundary.  Kaimanovich shows
  in~\cite{kaimanovich:ntexit} that the group of finitely supported
  permutations of a countable set admits a symmetric measure with
  non-trivial boundary. This provides an example of a measure with
  non-trivial boundary on the group $H$ from
  Theorem~\ref{thm:embedSinfty}. However, the measures obtained in
  this way are degenerate.

  It can be shown that the group $H$ considered in the proof of
  Theorem~\ref{thm:embedSinfty}, as well as other groups constructed
  in this paper, also admit non-degenerate measures with non-trivial
  Poisson-Furstenberg boundary. We will study random walks on
  permutational extensions, not restricted to those considered in this
  paper, in~\cite{bartholdi-e:rwperm}.  Some of the groups that
  we will be treated in~\cite{bartholdi-e:rwperm} lead to new
  phenomena in boundary behavior.
\end{remark}

\begin{remark}
  After this paper was completed, further developments on the growth
  of groups appeared in the
  preprints~\cites{brieussel:growth,kassabov-pak:growth,bartholdi-erschler:givengrowth}.
\end{remark}

\end{document}